\documentclass[preprint,12pt]{elsarticle}
\usepackage[cp1251]{inputenc}

\usepackage{a4wide,amsfonts, amsmath, amscd}
\usepackage[psamsfonts]{amssymb}

\usepackage{amssymb}
 \usepackage{graphicx}
\usepackage{epsfig}

\usepackage{amsthm}

\newcommand{\sysn}{\left\{\begin{array}{rcl}}
\newcommand{\sysk}{\end{array}\right.}

\newtheorem{theorem}{Theorem}[section]
\newtheorem{lemma}[theorem]{Lemma}

\newtheorem{proposition}[theorem]{Proposition}
\theoremstyle{definition}
\newtheorem{definition}[theorem]{Definition}

\newtheorem{corollary}[theorem]{Corollary}

\journal{...}

\begin{document}

\title{Fr\'{e}chet-Urysohn property of quasicontinuous functions}

\author{Alexander V. Osipov}

\address{Krasovskii Institute of Mathematics and Mechanics, \\ Ural Federal
 University, Ural State University of Economics, Yekaterinburg, Russia}

\ead{OAB@list.ru}

\begin{abstract} The aim of this paper is to study  the Fr\'{e}chet-Urysohn property of the space $Q_p(X,\mathbb{R})$ of
real-valued quasicontinuous functions, defined on a Hausdorff
space $X$, endowed with the pointwise convergence topology.

It is proved that under {\it Suslin's Hypothesis}, for an open
Whyburn space $X$, the space $Q_p(X,\mathbb{R})$ is
Fr\'{e}chet-Urysohn if and only if $X$ is countable. In
particular, it is true in the class of  first-countable regular
spaces $X$.

In $ZFC$, it is proved that for a metrizable space $X$, the space
$Q_p(X,\mathbb{R})$ is Fr\'{e}chet-Urysohn if and only if $X$ is
countable.
\end{abstract}



\maketitle 


\section{Introduction}

\medskip

A function $f$ from a topological space $X$ into $\mathbb{R}$ is
{\it quasicontinuous}, $f\in Q(X,\mathbb{R})$, if for every $x\in
X$ and open sets $U\ni x$ and $V\ni f(x)$ there exists a nonempty
open $W\subseteq U$ with $f(W)\subseteq V$.

The condition of quasicontinuity can be found in the paper of R.
Baire \cite{Bai} in study of continuity point of separately
continuous functions from $\mathbb{R}^2$ into $\mathbb{R}$. The
formal definition of quasicontinuity  were introduced by Kempisty
in 1932 in \cite{Kem}. Quasicontinuous functions were studied in
many papers, see for examples \cite{Bor,HM,Hol1,Hol2,Hol3,Hol5},
\cite{Hol7,Neo,Mar} and other. They found applications in the
study of topological groups \cite{Bou,Moo1,Moo3}, in the study of
dynamical systems \cite{Cr}, in the the study of CHART groups
\cite{Moo} and also used in the study of extensions of densely
defined continuous functions \cite{Hol6} and of extensions to
separately continuous functions on the product of pseudocompact
spaces \cite{Rez}, etc.

\medskip

Levine \cite{Lev} studied quasicontinuous maps under the name of
semi-continuity using the terminology of semi-open sets. A subset
$A$ of $X$ is {\it semi-open} if $A\subset \overline{Int(A)}$. A
function $f:X\rightarrow Y$ is called {\it semi-continuous} if
$f^{-1}(V)$ is semi-open in $X$ for every open
 set $V$ of $Y$. A map $f:X\rightarrow \mathbb{R}$ is quasicontinuous if and only
if $f$ is semi-continuous \cite{Lev}.

\medskip

Let $X$ be a Hausdorff  topological space, $Q(X,\mathbb{R})$ be
the space of all  quasicontinuous functions on $X$ with values in
$\mathbb{R}$ and $\tau_p$ be the pointwise convergence topology.
Denote by $Q_p(X,\mathbb{R})$ the topological space
$(Q(X,\mathbb{R}), \tau_p)$.

\medskip

A subset $U$ of a topological space $X$ is called a {\it regular
open set} or an {\it open domain} if $U=Int\overline{U}$ holds. A
subset $F$ of a topological space $X$ is called a {\it regular
closed } set or a {\it closed domain } if $F =\overline{IntF}$
holds. The family of regular open sets of $(X, \tau )$ is not a
topology. But it is a base for a topology $\tau_s$ called the {\it
semi-regularization} of $\tau$. If $\tau_s=\tau$, then $(X, \tau)$
is called {\it semi-regular} (or {\it quasi-regular}).

In (\cite{Nj}, Corollary 1), it is proved that a semi-regular
topology is the coarsest topology of its $\alpha$-class. Note that
all topologies of a given $\alpha$-class on $X$ determine the same
class of qusicontinuous mappings into an arbitrary topological
space (Proposition 9, \cite{Nj}). Since a Hausdorff topology
$\tau$ has a Hausdorff semi-regularization $\tau_s$ and
$Q_p((X,\tau),\mathbb{R})=Q_p((X,\tau_s),\mathbb{R})$, we can
further assume that  $X$ is a {\it Hausdorff semi-regular space}.

\medskip

In this paper we study the Fr\'{e}chet-Urysohn property of the
space $Q_p(X,\mathbb{R})$.

\section{Preliminaries}

Let us recall some properties and introduce new property of a
topological space $X$.

(1) A space $X$ is {\it Fr\'{e}chet-Urysohn} provided that for
every $A\subset X$ and $x\in \overline{A}$ there exists a sequence
in $A$ converging to $x$.


(2) A space $X$ is said to be {\it Whyburn} if $A\subset X$ and
$p\in \overline{A}\setminus A$ imply that there is a subset
$B\subseteq A$ such that $\overline{B}=B\cup\{x\}$.

(3) A space $X$ is said to be {\it $k$-Fr\'{e}chet-Urysohn} if for
every open subset $U$ of $X$ and every $x\in \overline{U}$, there
exists a sequence $(x_n)_{n\in \mathbb{N}}\subset U$ converging to
$x$.

\begin{definition} A topological space $X$ is called {\it open
Whyburn} if for every open set $A\subset X$ and every $x\in
\overline{A}\setminus A$ there is an open set $B\subseteq A$ such
that $\overline{B}\setminus A=\{x\}$.
\end{definition}

Let $X$ be a Tychonoff topological space, $C(X,\mathbb{R})$ be the
space of all  continuous functions on $X$ with values in
$\mathbb{R}$ and $\tau_p$ be the pointwise convergence topology.
Denote by $C_p(X,\mathbb{R})$ the topological space
$(C(X,\mathbb{R}), \tau_p)$.

\medskip
Let us recall that a cover $\mathcal{U}$ of a set $X$ is called

$\bullet$ an {\it $\omega$-cover} if each finite set $F\subseteq
X$ is contained in some $U\in \mathcal{U}$;

$\bullet$ a {\it $\gamma$-cover} if for any $x\in X$ the set
$\{U\in \mathcal{U}: x\not\in U\}$ is finite.

\medskip

A topological space $X$ is called a {\it $\gamma$-space} if each
$\omega$-cover $\mathcal{U}$ of $X$ contains a $\gamma$-subcover
of $X$. $\gamma$-Spaces were introduced by Gerlits and Nagy in
\cite{GN} and are important in the theory of function spaces as
they are exactly those $X$ for which the space $C_p(X,\mathbb{R})$
has the Fr\'{e}chet-Urysohn property \cite{GN2}.

\medskip

Clear that $C_p(X,\mathbb{R})$ is a subspace of
$Q_p(X,\mathbb{R})$. Thus, if $Q_p(X,\mathbb{R})$ is
Fr\'{e}chet-Urysohn then $C_p(X,\mathbb{R})$ is
Fr\'{e}chet-Urysohn, too. Hence, the property Fr\'{e}chet-Urysohn
of $Q_p(X,\mathbb{R})$ for a Tychonoff space $X$ implies that $X$
is a $\gamma$-space.



A set $A$ is called {\it minimally bounded} with respect to the
topology $\tau$ in a topological space $(X,\tau)$ if
$\overline{Int A}\supseteq A$ and $Int\overline{A}\subseteq A$
(\cite{1}, p.101). Clearly this means $A$ is semi-open and
$X\setminus A$ is semi-open. In the case of {\it open} sets,
minimal boundedness coincides with regular openness.

Note that if $U$ is a  minimally bounded (e.g. regular open) set
of $X$ such that $U$ is not dense subset in $X$ and $B\subset
\overline{U}\setminus U$ then there is a quasicontinuous function
$f:X\rightarrow \mathbb{R}$ such that $f(U\cup B)=0$ and
$f(X\setminus (U\cup B))=1$ (see Lemma 4.2 in \cite{Hol5}).

\begin{proposition}\label{pr11} Let $Q_p(X,\mathbb{R})$ be a
Fr\'{e}chet-Urysohn space. Then $\overline{W}\setminus W$ is
countable for every minimally bounded set $W$ of $X$.
\end{proposition}

\begin{proof} Let $W$ be a minimally bounded set $W$ of $X$.
Note that $W\cup B$ is a minimally bounded set in $X$ for any
$B\subseteq \overline{W}\setminus W$.

Let $M_K=W\cup (\overline{W}\cap K)$ for each $K\in
[X]^{<\omega}$.

Suppose that $D=\overline{W}\setminus W$ is uncountable.

Consider the set $C=\{f_{K}: K\in [X]^{<\omega}\}$ of
quasicontinuous functions $f_{K}$ where

 $$ f_{K}:=\left\{
\begin{array}{lcl}
0 \, \, \, \, \, \, \, \, \, \, on \, \, \, \,  M_K\\
1 \, \, \, on \, \, \, X\setminus M_K.\\
\end{array}
\right.
$$

Let $$ g:=\left\{
\begin{array}{lcl}
0 \, \, \, \, \, \, \, \, \, \, on \, \, \, \,  \overline{W}\\
1 \, \, \, on \, \, \, X\setminus \overline{W}.\\
\end{array}
\right.
$$

Note that $g\in Q_p(X,\mathbb{R})$ and $g\in \overline{C}$. Since
$Q_p(X,\mathbb{R})$ is Fr\'{e}chet-Urysohn there is a sequence
$\{f_{K_i}: i\in \mathbb{N}\}\subset C$ such that
$f_{K_i}\rightarrow g$ ($i\rightarrow \infty$). Since $D$ is
uncountable, there is $z\in D\setminus \bigcup\limits_i K_i$.
Consider $[z,(-\frac{1}{2},\frac{1}{2})]=\{f\in Q_p(X,\mathbb{R}):
f(z)\in (-\frac{1}{2},\frac{1}{2})\}$. Note that $g\in
[z,(-\frac{1}{2},\frac{1}{2})]$ and $f_{K_i}\notin
[z,(-\frac{1}{2},\frac{1}{2})]$ for any $i\in \mathbb{N}$
($f_{K_i}(z)=1$ for every $i\in \mathbb{N}$), it is a
contradiction.

\end{proof}




\section{Main results}

\begin{lemma}\label{lem11} Let $X$ be an open Whyburn space such that $Q_p(X,\mathbb{R})$ is Fr\'{e}chet-Urysohn. Then every nowhere subset in $X$ is
countable.
\end{lemma}

\begin{proof} Since the closure of a nowhere dense subset in $X$ is a nowhere dense set, we can
consider only closed nowhere dense sets in $X$.

   Assume that $A$ is an uncountable closed nowhere dense set in
$X$. Since $X$ is open Whyburn, for every point $a\in A$ there is
a regular open set $O_a\subseteq X\setminus A$ such that
$\overline{O_a}\setminus (X\setminus A)=\{a\}$.

 For every a finite subset $K$ of $X$ we consider the
set

$M_K=S_K\cup \bigcup\{O_a\cup\{a\}: a\in K\cap A\}$ where $S_K$ is
a regular open set such that $K\cap (X\setminus A)\subseteq
S_{K}\subseteq X\setminus A$. Note that $M_K$ is minimally bounded
set in $X$.

Consider the set $S=\{f_{K}: K\in [X]^{<\omega}\}$ of
quasicontinuous functions $f_{K}$ where

 $$ f_{K}:=\left\{
\begin{array}{lcl}
0 \, \, \, \, \, \, \, \, \, \, on \, \, \, \,  M_K\\
1 \, \, \, on \, \, \, X\setminus M_K.\\
\end{array}
\right.
$$

Note that ${\bf 0}\in \overline{S}$ where ${\bf 0}$ denote the
constant function on $X$ with value 0. Since $Q_p(X,\mathbb{R})$
is Fr\'{e}chet-Urysohn, there is a sequence $\{f_{K_i}: i\in
\mathbb{N}\}\subset S$ such that $f_{K_i}\rightarrow {\bf 0}$
($i\rightarrow \infty$). Since $A$ is uncountable, there is $z\in
A\setminus \bigcup\limits_i K_i$. Consider
$[z,(-\frac{1}{2},\frac{1}{2})]=\{f\in Q_p(X,\mathbb{R}): f(z)\in
(-\frac{1}{2},\frac{1}{2})\}$. Note that ${\bf 0}\in
[z,(-\frac{1}{2},\frac{1}{2})]$ and $f_i\notin
[z,(-\frac{1}{2},\frac{1}{2})]$ for any $i\in \mathbb{N}$
($f_i(z)=1$ for every $i\in \mathbb{N}$), it is a contradiction.
\end{proof}

\begin{definition}(\cite{Kun})
A Hausdorff space $X$ is called a {\it Lusin space} ({\it in the
sense of Kunen})  if

(a) Every nowhere dense set in $X$ is countable;

(b) $X$ has at most countably many isolated points;

(c) $X$ is uncountable.
\end{definition}

\begin{theorem}\label{th1} Let $X$ be an uncountable open Whyburn space such that $Q_p(X,\mathbb{R})$ is Fr\'{e}chet-Urysohn.
Then $X$ is a Lusin space.
\end{theorem}

\begin{proof} By Lemma \ref{lem11}, it is enough to prove that $X$ has at most countably many isolated points.

 Assume that $X$ has uncountable many isolated points $D$.

Let $D=D_1\cup D_2$ where $D_1\cap D_2=\emptyset$ and $|D_i|>
\aleph_0$ for $i=1,2$. Consider the set $W=Int\overline{D_1}$.
Clear that $\overline{W}\cap D_2=\emptyset$. By Lemma \ref{lem11},
$|W\setminus D_1|\leq \omega$.

Since $X$ is open Whyburn, for every point $d\in W\setminus D_1$
there is an open subset $O_d\subseteq D_1$ such that
$\overline{O_d}\setminus D_1=\{d\}$.

(a) Suppose that for every point $d\in W\setminus D_1$ there is a
neighborhood $V_d$ of $d$ such that  $|O_d\cap V_d|\leq \omega$.
Let $W_d=O_d\cap V_d$. Then $\overline{W_d}\setminus D_1=\{d\}$,
$W_d\subset D_1$ and $|W_d|\leq \omega$.

For every a finite subset $K$ of $W$ we consider the set

$P_K=\bigcup\{\{d\}: d\in K\cap D_1\}\cup \bigcup\{\overline{W_d}:
d\in K\cap (W\setminus D_1)\}$.

Consider the set $C=\{g_{K}: K\in [W]^{<\omega}\}$ of
quasicontinuous functions $g_{K}$ where

 $$ g_{K}:=\left\{
\begin{array}{lcl}
0 \, \, \, \, \, \, \, \, \, \, on \, \, \, \,  P_K\\
1 \, \, \, on \, \, \, X\setminus P_K.\\
\end{array}
\right.
$$

Let $$ g:=\left\{
\begin{array}{lcl}
0 \, \, \, \, \, \, \, \, \, \, on \, \, \, \,  W\\
1 \, \, \, on \, \, \, X\setminus W.\\
\end{array}
\right.
$$

Note that $g\in Q_p(X,\mathbb{R})$ and $g\in \overline{C}$. Since
$Q_p(X,\mathbb{R})$ is Fr\'{e}chet-Urysohn there is a sequence
$\{g_{K_i}: i\in \mathbb{N}\}\subset C$ such that
$g_{K_i}\rightarrow g$ ($i\rightarrow \infty$). Since $D_1$ is
uncountable, there is $z\in D_1\setminus \bigcup\limits_i
P_{K_i}$. Consider $[z,(-\frac{1}{2},\frac{1}{2})]=\{f\in
Q_p(X,\mathbb{R}): f(z)\in (-\frac{1}{2},\frac{1}{2})\}$. Note
that $g\in [z,(-\frac{1}{2},\frac{1}{2})]$ and $g_i\notin
[z,(-\frac{1}{2},\frac{1}{2})]$ for any $i\in \mathbb{N}$
($g_i(z)=1$ for every $i\in \mathbb{N}$), it is a contradiction.

(b) Suppose that there is a point $d\in W\setminus D_1$ such that
$|O_d\cap V_d|>\omega$ for every neighborhood $V_d$ of $d$. Let
$O_d=O_1\cup O_2$ such that $O_1\cap O_2=\emptyset$ and
$|O_i|>\omega$ for $i=1,2$.

There are two cases:

(1) $V_d\cap O_i\neq \emptyset$ for every neighborhood $V_d$ of
$d$ and $i=1,2$;

(2) $V_d\cap O_i=\emptyset$ for some neighborhood $V_d$ of $d$ and
some $i=1,2$.

Suppose that the case (2) is true for $i=1$. Note that in this
case $d\in \overline{O_2}$.

 Then, for cases
(1) and (2), we consider the set $C=\{g_{K}: K\in
[O_1]^{<\omega}\}$ of continuous functions $g_{K}$ where

 $$ g_{K}:=\left\{
\begin{array}{lcl}
0 \, \, \, \, \, \, \, \, \, \, on \, \, \, \,  K\\
1 \, \, \, on \, \, \, X\setminus K.\\
\end{array}
\right.
$$

Let $$ g:=\left\{
\begin{array}{lcl}
0 \, \, \, \, \, \, \, \, \, \, on \, \, \, \,  O_1\\
1 \, \, \, on \, \, \, X\setminus O_1.\\
\end{array}
\right.
$$

Note that $g\in Q_p(X,\mathbb{R})$ (for cases: (1) and (2)) and
$g\in \overline{C}$. Since $Q_p(X,\mathbb{R})$ is
Fr\'{e}chet-Urysohn there is a sequence $\{g_{K_i}: i\in
\mathbb{N}\}\subset C$ such that $g_{K_i}\rightarrow g$
($i\rightarrow \infty$). Since $O_1$ is uncountable, there is
$z\in O_1\setminus \bigcup\limits_i K_i$. Consider
$[z,(-\frac{1}{2},\frac{1}{2})]=\{f\in Q_p(X,\mathbb{R}): f(z)\in
(-\frac{1}{2},\frac{1}{2})\}$. Note that $g\in
[z,(-\frac{1}{2},\frac{1}{2})]$ and $g_i\notin
[z,(-\frac{1}{2},\frac{1}{2})]$ for any $i\in \mathbb{N}$
($g_i(z)=1$ for every $i\in \mathbb{N}$), it is a contradiction.
\end{proof}

Let $I(X)$ denote the set of isolated points of $X$. Note that a
Lusin space has at most countably many isolated points.

\begin{corollary}{\it Let $X$ be an open Whyburn space such that $I(X)$ is an uncountable dense subset in $X$.
Then $Q_p(X,\mathbb{R})$ is not Fr\'{e}chet-Urysohn}.
\end{corollary}

\begin{proposition} Let $X$ be a $k$-Fr\'{e}chet-Urysohn regular
space with countable pseudocharacter. Then $X$ is open Whyburn.
\end{proposition}

\begin{proof} Let $x\in \overline{U}\setminus U$ for an open set
$U$ in $X$. Since $X$ is $k$-Fr\'{e}chet-Urysohn, there is a
sequence $\{x_n: n\in\mathbb{N}\}\subset U$ such that
$x_n\rightarrow x$ ($n\rightarrow \infty$). Since $X$ is a regular
space with countable pseudocharacter, there is a sequence $\{V_i:
i\in \mathbb{N}\}$ of open neighborhoods of $x$ such that $\bigcap
V_i=\{x\}$ and $\overline{V_{i+1}}\subset V_i$ for each $i\in
\mathbb{N}$. We can assume that $x_i\in V_i\setminus
\overline{V_{i+1}}$. Let $W_i$ be a neighborhood of $x_i$ such
that $\overline{W_i}\subset U\cap (V_i\setminus
\overline{V_{i+1}})$. Then $W=\bigcup\{W_i: i\in
\mathbb{N}\}\subset U$ and $\overline{W}\setminus U=\{x\}$.

\end{proof}

\begin{corollary}\label{cor1} {\it Let $X$ be an uncountable $k$-Fr\'{e}chet-Urysohn (Fr\'{e}chet-Urysohn) regular
space with countable pseudocharacter such that $Q_p(X,\mathbb{R})$
is  Fr\'{e}chet-Urysohn. Then $X$ is a Lusin space}.
\end{corollary}

In particular, if $X$ is an uncountable first-countable regular
space such that $Q_p(X,\mathbb{R})$ is Fr\'{e}chet-Urysohn then
$X$ is a Lusin space.

\medskip

Note that if $X$ is Tychonoff and $C_p(X,\mathbb{R})$ is
Fr\'{e}chet-Urysohn then $C_p(X^2,\mathbb{R})$ is
Fr\'{e}chet-Urysohn \cite{GN2}. However, this is not true for
quasicontinuous functions.

\begin{corollary} {\it Let $X$ be an uncountable first-countable regular space such that $Q_p(X,\mathbb{R})$ is Fr\'{e}chet-Urysohn.
Then $Q_p(X^2,\mathbb{R})$ is not Fr\'{e}chet-Urysohn space}.
\end{corollary}

\begin{proof}
Since $\mathbb{R}^{\kappa}$ is not Fr\'{e}chet-Urysohn for any
$\kappa\geq\omega_1$, $X$ is not discrete space provided that $X$
is an uncountable and $Q_p(X,\mathbb{R})$ is Fr\'{e}chet-Urysohn.
Clear that $X^2=X\times X$ is not Lusin space provided that $X$ is
a Lusin space and $X$ with a non-isolated point.
\end{proof}
\medskip

By Theorem \ref{th1} and results in \cite{Kun} (Lemmas 1.2 and
1.5), we get that if $X$ is an uncountable open Whyburn Hausdorff
semi-regular space such that $Q_p(X,\mathbb{R})$ is
Fr\'{e}chet-Urysohn then $X$ is hereditarily Lindel\"{o}f (hence,
$X$ is perfect normal (see 3.8.A. in \cite{Eng})) and $X$ is
zero-dimensional.

\medskip

Since a Lusin space $X$ is hereditarily Lindel\"{o}f and
Hausdorff, it has cardinality at most $\mathfrak{c}=2^{\omega}$
(de Groot, \cite{Gro}).

\begin{corollary} {\it Let $X$ be an open Whyburn space of  cardinality $> \mathfrak{c}$.
Then $Q_p(X,\mathbb{R})$ is not Fr\'{e}chet-Urysohn space}.
\end{corollary}

In particular, if $X$ is first-countable regular space of
cardinality $> \mathfrak{c}$ then $Q_p(X,\mathbb{R})$ is not
Fr\'{e}chet-Urysohn space.

\medskip

Let us note however that Kunen (Theorem 0.0. in \cite{Kun}) has
shown that under {\it Suslin's Hypothesis} (${\bf SH}$) there are
no Lusin spaces at all. K.Kunen proved that under $ {\bf
MA}(\aleph_1,\aleph_0$-centred$)$ there is a Lusin space if and
only if there is a Suslin line.

\medskip

The Suslin Hypothesis is neither provable nor refutable in ${\bf
ZFC}$, even if we assume ${\bf CH}$ or ${\bf \neg CH}$. A typical
model of ${\bf ZFC+\neg SH}$ is the G\"{o}del constructible
universe ${\bf L}$, while a typical model of ${\bf ZFC+SH}$ is the
Solovay-Tennenbaum model of ${\bf ZFC+MA(\aleph_1)}$ (see p.266 in
\cite{handbook}).

\begin{theorem} $({\bf SH}).$ Let $X$ be an open Whyburn space. The space $Q_p(X,\mathbb{R})$ is Fr\'{e}chet-Urysohn if and only if
$X$ is countable.

\end{theorem}

In particular, for first-countable regular spaces, we have the
following corollary.
\medskip
\begin{corollary} $({\bf SH}).$ {\it Let $X$ be a first-countable regular space. The space $Q_p(X,\mathbb{R})$ is Fr\'{e}chet-Urysohn if and only if
$X$ is countable}.
\end{corollary}

However, the following result  holds in ${\bf ZFC}$.

\begin{theorem}\label{cor210}{\it Let $X$ be a metrizable space. The space $Q_p(X,\mathbb{R})$ is
Fr\'{e}chet-Urysohn if and only if $X$ is countable}.
\end{theorem}

\begin{proof} Note that a Lusin subspace of a metrizable space is a
Lusin set:{\it an uncountable subset of $\mathbb{R}$ that meets
every nowhere dense set in a countable set}. Hence, if
$Q_p(X,\mathbb{R})$ is Fr\'{e}chet-Urysohn then $X$ is a Lusin set
and it is a $\gamma$-space. But any $\gamma$-space $X\subset
\mathbb{R}$ is {\it always first category} (see Definition in
\cite{Ser})  and a Lusin set is not {\it always first category}
(p. 159 in \cite{GN}). Hence, $X$ is countable.

If $X$ is countable then $Q_p(X,\mathbb{R})$ is first countable
(Theorem 4.1 in \cite{Hol5}) and, hence, $Q_p(X,\mathbb{R})$ is
Fr\'{e}chet-Urysohn.

\end{proof}

\section{Selection principle $S_1$ and Fr\'{e}chet-Urysohn at the point ${\bf 0}$}

Let $\mathcal{A}$ and $\mathcal{B}$ be collections of covers of a
topological space $X$.

\medskip

The symbol $S_1(\mathcal{A}, \mathcal{B})$ denotes the selection
principle that for each sequence $\langle \mathcal{U}_n : n \in
\mathbb{N} \rangle$ of elements of $\mathcal{A}$ there exists a
sequence $\langle U_n : n \in \mathbb{N} \rangle$ such that for
each $n$, $U_n \in \mathcal{U}_n$ and $\{U_n : n \in \mathbb{N}\}
\in \mathcal{B}$ (see \cite{H1}).

\medskip

In this paper $\mathcal{A}$ and $\mathcal{B}$ will be collections
of the following  covers of a space $X$:


$\Omega$ : the collection of open $\omega$-covers of $X$.

$\Gamma$ : the collection of open $\gamma$-covers of $X$.

$\Omega^s$ : the collection of minimally bounded  $\omega$-covers
of $X$.

$\Gamma^s$ : the collection of minimally bounded  $\gamma$-covers
of $X$.

\medskip

In \cite{GN2}, it is proved that $C_p(X,\mathbb{R})$ is
Fr\'{e}chet-Urysohn if and only if $X$ has the property
$S_1(\Omega,\Gamma)$.

\medskip

\begin{lemma}\label{lem121} Let $Q_p(X,\mathbb{R})$ be Fr\'{e}chet-Urysohn at the point ${\bf 0}$. Then
$X$ has the property $S_1(\Omega^s, \Gamma^s)$.

\end{lemma}

\begin{proof} Let $\{\mathcal{V}_i: i\in \mathbb{N}\}$ be a family
of minimally bounded  $\omega$-covers of $X$. For each $i\in
\mathbb{N}$, we consider the family $A_i=\{f_{i,V}\in
Q_p(X,\mathbb{R}): V\in \mathcal{V}_i\}$ such that
$f_{i,V}(V)=\frac{1}{i}$ and $f_{i,V}(X\setminus V)=1$ for
$V\in\mathcal{V}_i$. Let $A=\bigcup A_i$. Then ${\bf 0}\in
\overline{A}\setminus A$. Since $Q_p(X,\mathbb{R})$ is
Fr\'{e}chet-Urysohn at point ${\bf 0}$, there is a sequence
$\{f_{i,V_i}: i\in \mathbb{N}\}$ such that $f_{i,V_i}\in A$ for
each $i\in \mathbb{N}$ and  $f_{i,V_i}\rightarrow {\bf 0}$
($i\rightarrow\infty$). Note that $\{V_i: i\in \mathbb{N}\}$ is a
minimally bounded $\gamma$-cover of $X$.
\end{proof}

\begin{theorem} Let $X$ be an open Whyburn space. The space $Q_p(X,\mathbb{R})$ is Fr\'{e}chet-Urysohn at the point ${\bf 0}$ if and only if
$X$ has the property $S_1(\Omega^s, \Gamma^s)$.
\end{theorem}

\begin{proof} By Lemma \ref{lem121}, it is enough to prove a
sufficient condition.

Let ${\bf 0}\in \overline{A}\setminus A$ for some set $A\subset
Q_p(X,\mathbb{R})$. For each $i\in \mathbb{N}$, we consider the
set $\mathcal{U}_i=\{f^{-1}(-\frac{1}{i},\frac{1}{i}): f\in A\}$.
Clear that $\mathcal{U}_i$ is a semi-open $\omega$-cover of $X$
for each $i\in \mathbb{N}$.

Let $U\in \mathcal{U}_i$. Since $X$ is an open Whyburn
semi-regular space, for each finite subset $K$ of $U$, there is a
minimally bounded set $V_{K,U,i}$ such that $K\subset
V_{K,U,i}\subset U$. Thus, the family $\mathcal{V}_i=\{V_{K,U,i}:
K\in[U]^{<\omega}$ and $U\in \mathcal{U}_i\}$ is a minimally
bounded $\omega$-cover of $X$ for each $i\in \mathbb{N}$. Since
$X$ has the property $S_1(\Omega^s, \Gamma^s)$ there exists a
sequence $(V_{K_i,U_i,i} : i \in \mathbb{N})$ such that for each
$i$, $V_{K_i,U_i,i} \in \mathcal{V}_i$ and $\{V_{K_i,U_i,i} : i
\in \mathbb{N}\}$ is a minimally bounded $\gamma$-cover of $X$.
Then the sequence $(f_i: U_i=f_i^{-1}(-\frac{1}{i},\frac{1}{i}),
i\in \mathbb{N})\rightarrow {\bf 0}$ $(i\rightarrow \infty)$.

\end{proof}

\section{Examples}

Similarly the proof of Proposition \ref{pr11}, we get the
following result.

\begin{proposition}\label{prop1} Let $X$ be a space with a dense subset $D$ of isolated points such that $D=D_1\cup D_2$ where $\overline{D_1}=X\setminus D_2$ and
$\overline{D_2}=X\setminus D_1$ and let $Q_p(X,\mathbb{R})$ be a
Fr\'{e}chet-Urysohn space. Then $D$ is countable.
\end{proposition}

\begin{proposition} {\it There is a compact space $X$ such that $C_p(X,\mathbb{R})$ is Fr\'{e}chet-Urysohn, but $Q_p(X,\mathbb{R})$ is not}.
\end{proposition}

\begin{proof}
Let $X=\omega_1+1$. Here $\omega_1+1$ is the space
$\{\alpha:\alpha\leq \omega_1\}$ with the order topology. By
Proposition \ref{prop1}, $Q_p(X,\mathbb{R})$ is not
Fr\'{e}chet-Urysohn. It well known that $C_p(Y,\mathbb{R})$ is
Fr\'{e}chet-Urysohn for a compact space $Y$ if and only if $Y$ is
scattered \cite{GN}. Hence, $C_p(X,\mathbb{R})$ is
Fr\'{e}chet-Urysohn.
\end{proof}

\begin{proposition} There is an uncountable separable metrizable space
$X$ such that $C_p(X,\mathbb{R})$ is Fr\'{e}chet-Urysohn, but
$Q_p(X,\mathbb{R})$ is not.
\end{proposition}

By Corollary \ref{cor210}, it is enough consider any uncountable
$\gamma$-space $X\subset \mathbb{R}$.

\begin{proposition} There is an uncountable $T_1$-space $X$ such that
$Q_p(X,\mathbb{R})$ is Fr\'{e}chet-Urysohn.
\end{proposition}

Let $X$ be an uncountable set with the cofinite (or co-countable)
topology. Then $Q_p(X,\mathbb{R})$ is homeomorphic to $\mathbb{R}$
 because any quasicontinuous function on $X$
is a constant function, and, hence, $Q_p(X,\mathbb{R})$ is
Fr\'{e}chet-Urysohn.

\section{Open questions}

{\bf Question 1.} Suppose that $X$ is a (first-countable, regular)
submetrizable and $Q_p(X,\mathbb{R})$ is Fr\'{e}chet-Urysohn. Is
the space $X$ countable?

\medskip

{\bf Question 2.} Suppose that $X$ is an open Whyburn $T_2$
semi-regular space and $Q_p(X,\mathbb{R})$ is Fr\'{e}chet-Urysohn.
Is the space $X$ countable?

\medskip

{\bf Question 3.} Suppose that $X$ is a $T_2$ space and
$Q_p(X,\mathbb{R})$ is Fr\'{e}chet-Urysohn. Is the space $X$
Lusin?

\medskip

{\bf Question 4.} Suppose that $X$ is a Lusin space and a
$\gamma$-space. Is the space $X$ countable?

\medskip

{\bf Question 5.} Suppose that a (an open Whyburn) space $X$ has
the property $S_1(\Omega^s, \Gamma^s)$.  Will the space
$Q_p(X,\mathbb{R})$ have the Fr\'{e}chet-Urysohn property at each
point?

\bibliographystyle{model1a-num-names}
\bibliography{<your-bib-database>}

\end{document}